\documentclass[11pt, reqno]{amsart}
\usepackage{amsmath, amsthm, amscd, amsfonts, amssymb, graphicx, color}
\usepackage[bookmarksnumbered, colorlinks, plainpages]{hyperref}
     \makeatletter
     \def\section{\@startsection{section}{1}%
      \z@{.7\linespacing\@plus\linespacing}{.5\linespacing}%
     {\bfseries%\normalfont\scshape
     \centering
     }}
     \def\@secnumfont{\bfseries}
     \makeatother
\newtheorem{thm}{Theorem}
\newtheorem{lem}[thm]{Lemma}
\newtheorem{prop}[thm]{Proposition}
\newtheorem{cor}[thm]{Corollary}
\theoremstyle{definition}

\newcommand{\vertiii}[1]{{\left\vert\kern-0.25ex\left\vert\kern-0.25ex\left\vert
#1 \right\vert\kern-0.25ex\right\vert\kern-0.25ex\right\vert}}

\def \dim   {\text {\rm dim}}

\def \diag  {\text {\rm diag}}
\def \span   {\text {\rm span}}

\def \min   {\text {\rm min}}

\def \diag  {\text {\rm diag}}
\def \lim   {\text {\rm lim}}

\def \wt    {\widetilde}

\def \qed   {\hfill \vrule height6pt width 6pt depth 0pt}

\begin{document}

\title[]{Sums and products of symplectic eigenvalues}

\author{Tanvi Jain}

\address{Indian Statistical Institute, New Delhi 110016, India}

\email{tanvi@isid.ac.in}
%\dedicatory{Dedicated to Professor Rajendra Bhatia}
\date{\today}
\begin{abstract}
For every $2n\times 2n$ real positive definite matrix $A,$ there exists a real symplectic matrix $M$ such that $M^TAM=\diag(D,D),$ where $D$ is the $n\times n$ positive diagonal matrix with diagonal entries $d_1(A)\le \cdots\le d_n(A).$ The numbers $d_1(A),\ldots,d_n(A)$ are called the symplectic eigenvalues of $A.$
We derive analogues of Wielandt's extremal principle and multiplicative Lidskii's inequalities for symplectic eigenvalues.
\end{abstract}



\maketitle
\section{Introduction}
Let $J$ be the $2n\times 2n$ matrix
\begin{equation}\label{eq1}
J=\begin{bmatrix}O & I_n\\
-I_n & O\end{bmatrix},
\end{equation}
where $I_n$ is the $n\times n$ identity matrix.
Consider the antisymmetric bilinear form $(\cdot,\cdot)$ on $\mathbb{R}^{2n}$ given by
$$(x,y)=\langle x,Jy\rangle,\ x,y\in\mathbb{R}^{2n}.$$
We call this bilinear form to be the {\it symplectic inner product} on $\mathbb{R}^{2n}.$
A $2n\times 2n$ real matrix that preserves the symplectic inner product is called a {\it symplectic matrix}, i.e., $M$ is a symplectic matrix if
$$\langle Mx,JMy\rangle=\langle x,Jy\rangle$$
for all $x,y$ in $\mathbb{R}^{2n}.$ This is equivalent to saying that $M$ is symplectic if and only if $M^TJM=J.$
We denote by $Sp(2n)$ the collection of all $2n\times 2n$ real symplectic matrices. This is a Lie group under matrix multiplication.
A basic theorem in symplectic matrix theory, generally known as {\it Williamson's theorem} tells us that for every $2n\times 2n$ real positive definite matrix $A,$
there exists a $2n\times 2n$ symplectic matrix $M$ such that
\begin{equation}
M^TAM=\begin{bmatrix}D & O\\
O & D\end{bmatrix},\label{eqi1}
\end{equation}
where $D$ is the $n\times n$ positive diagonal matrix with diagonal entries $d_1(A)\le\cdots\le d_n(A).$
These positive numbers are complete invariants for $A$ under the action of the symplectic group. We call them the {\it symplectic eigenvalues} of $A.$
Symplectic eigenvalues are important in different areas of mathematics and physics such as symplectic geometry, symplectic topology, both classical and quantum mechanics and quantum information. See, for instance, \cite{dms, dg, h, hz, sanders, k}.
In particular, applications of symplectic eigenvalues in the newer area of quantum information has driven much of the recent work on this topic \cite{dms, sanders, k, p, saf}.

For over a century, there has been an extensive interest in the study of variational principles for eigenvalues of Hermitian matrices and the relationships between the eigenvalues of $A,$ $B$ and $A+B.$
We refer the reader to Chapter III of \cite{rbh}, Chapter 5 of \cite{hj} and to the excellent exposition \cite{rbh1} for detailed accounts of these topics.
If $A$ is a $2n\times 2n$ real positive definite matrix, then $\imath A^{1/2}JA^{1/2}$ is a Hermitian matrix.
It can be seen that $A$ has symplectic eigenvalues $d_1,\ldots,d_n$ if and only if $\imath A^{1/2}JA^{1/2}$ has eigenvalues $\pm d_1,\ldots,\pm d_n.$
So, in principle, it could be possible to derive the properties of symplectic eigenvalues from the well-known properties of eigenvalues of Hermitian matrices.
But due to the complicated form of the matrix $\imath A^{1/2}JA^{1/2},$
it is often not feasible to do so.
The theory of symplectic eigenvalues also differs from the eigenvalue theory because of the differences in the symplectic group and the orthogonal group.
It is very remarkable that inspite of some major differences, we can find symplectic analogues to many of the classical results in eigenvalue theory \cite{bj,h,sanders,jm}.
%In this paper, we further derive symplectic analogues of Wielandt's principle %and some classical eigenvalue inequalities.
Two versions of symplectic maxmin principles, analogous to the Courant-Fischer-Weyl minmax principle can be found in the literature. See \cite{dg, bj, bj2}.
These principles are in turn used to give symplectic analogues of an interlacing theorem, monotonicity principle and Weyl's inequalities.
In \cite{h}, Hiroshima proved the inequalities $\sum_{j=1}^{k}d_j(A+B)\ge\sum_{j=1}^{k}\left(d_j(A)+d_j(B)\right),$ $1\le k\le n$;
and these inequalities were generalised to Lidskii's type inequalities in \cite{jm} by using the analyticity of symplectic eigenvalues.
A symplectic analogue of the Ky Fan maximum principle was proved in \cite{h}.
An important extremal principle that includes the Courant-Fischer-Weyl principle as well as the Ky Fan maximum principle
 for eigenvalues of Hermitian matrices is the Wielandt extremal principle.
This played an important role in the development of the celebrated Horn's conjecture.
It gives us an expression for arbitrary sum of $k$ eigenvalues of Hermitian matrices.
If $A$ is an $n\times n$ Hermitian matrix with eigenvalues $\lambda_1(A)\le\cdots \le \lambda_n(A),$
then for all $1\le k\le n$ and $1\le i_1<\cdots < i_k\le n,$
$$\sum\limits_{j=1}^{k}\lambda_{i_j}(A)={\underset{\underset{\dim\mathcal{W}_j=n-i_j+1}{\mathcal{W}\supset\cdots\supset\mathcal{W}_k}}{\max}}{\underset{\underset{\textrm{orthonormal}}{x_j\in\mathcal{W}_j}}{\min}}\sum\limits_{j=1}^{k}\langle x_j,Ax_j\rangle.$$
We give a symplectic analogue of the Wielandt principle.
A subset $\{x_1,\ldots,x_k,y_1,\ldots,y_k\}$ of $\mathbb{R}^{2n}$ is called {\it symplectically orthonormal} if
$$\langle x_i,Jx_j\rangle=\langle y_i,Jy_j\rangle=0\textrm{ and }\langle x_i,Jy_j\rangle=\delta_{ij}$$
for all $i,j=1,\ldots,k.$

\begin{thm}\label{thm_wndt}
Let $A$ be a $2n\times 2n$ real positive definite matrix with symplectic eigenvalues $d_1\le\cdots\le d_n.$
Let $1\le k\le n$ and $1\le i_1<\cdots<i_k\le n.$ Then
\begin{equation}
\sum\limits_{j=1}^{k}d_{i_j}={\underset{\underset{\dim\, \mathcal{W}_j=2n-i_j+1}{\mathcal{W}_1\supset\cdots\supset\mathcal{W}_k}}{\max}}{\underset{\underset{\underset{\textrm{orthonormal}}{\textrm{symplectically}}}{x_j,y_j\in\mathcal{W}_j}}{\min}}\sum\limits_{j=1}^{k}\frac{\langle x_j,Ax_j\rangle+\langle y_j,Ay_j\rangle}{2}.\label{eqwndt}
\end{equation}
\end{thm}

This gives us an alternative proof of the Lidskii type inequalities for symplectic eigenvalues that were first proved in \cite{jm}.

\begin{cor}\label{cor_lid}
Let $A$ and $B$ be two $2n\times 2n$ real positive definite matrices.
Suppose that $A,$ $B$ and $A+B$ have symplectic eigenvalues
$d_1(A)\le\cdots\le d_n(A),$ $d_1(B)\le\cdots\le d_n(B)$ and $d_1(A+B)\le\cdots\le d_n(A+B),$ respectively.
Then for all $k=1,\ldots,n$ and $1\le i_1<\cdots< i_k\le n,$
\begin{equation}
\sum\limits_{j=1}^{k}d_{i_j}(A+B)\ge \sum\limits_{j=1}^{k}d_{i_j}(A)+\sum\limits_{j=1}^{k}d_j(B).\label{eqlid}
\end{equation}
\end{cor}

We also study the multiplicative versions of Lidskii's inequalities for symplectic eigenvalues.
For two positive definite matrices $A$ and $B,$ their {\it geometric mean} is the positive definite matrix
$$A\# B=A^{1/2}\left(A^{-1/2}BA^{-1/2}\right)^{1/2}A^{1/2}.$$
This was introduced by Pusz and Woronowicz in \cite{pw} and have important connections with
various areas such as physics, computer science, matrix analysis and Riemannian geometry \cite{n}.
Suppose $A$ and $B$ are two $2n\times 2n$ real positive definite matrices.
Let $d_1(A)\le\cdots\le d_n(A),$ $d_1(B)\le\cdots\le d_n(B)$ and
$d_1(A\# B)\le\cdots\le d_n(A\# B)$ be the symplectic eigenvalues of $A,$ $B$ and $A\# B,$ respectively.
The following relationships were proved in Theorem 3 of \cite{bj}: For all $k=1,\ldots,n,$
$$\prod\limits_{j=1}^{k}d_j^2(A\# B)\ge\prod\limits_{j=1}^{k}d_j(A)d_j(B)$$
and $$\prod\limits_{j=k}^{n}d_j^2(A\# B)\le \prod\limits_{j=k}^{n}d_j(A)d_j(B).$$

As a consequence of our analysis, we can extend these inequalities to arbitrary products of symplectic eigenvalues
analogous to the multiplicative inequalities proved by Lidskii for eigenvalues of two positive definite matrices \cite{markus}.

\begin{thm}\label{thmmlid}
Let $A$ and $B$ be two $2n\times 2n$ real positive definite matrices.
Then for all $k=1,\ldots,n$ and $1\le i_1<\cdots< i_k\le n$
\begin{equation}
\prod\limits_{j=1}^{k}d_{i_j}(A)d_j(B)\le \prod\limits_{j=1}^{k}d_{i_j}^2(A\# B)\le \prod\limits_{j=1}^{k}d_{i_j}(A)d_{n-j+1}(B).\label{eqmlid}
\end{equation}
\end{thm}

In Section 2, we recall some basic facts on symplectic spaces and symplectic eigenvalues, and introduce some terminology and notation.
 We prove Theorem \ref{thm_wndt} and Corollary \ref{cor_lid} in Section 3.
In the process we give a third version of the maxmin principle for symplectic eigenvalues.
Theorem 3 is proved in Section 4.

\section{Preliminaries}

A subspace $\mathcal{V}$ of $\mathbb{R}^{2n}$ is called a {\it symplectic space} if for every $x\in\mathcal{V},$ there exists a $y\in\mathcal{V}$ such that $\langle x,Jy\rangle\ne 0.$ It can be verified that every symplectic space has even dimension. Let $\mathcal{V}$ be a symplectic subspace of $\mathbb{R}^{2n}$ and let $S$ be a nonempty subset of $\mathcal{V}.$ Then the set $$S^{\perp_s}=\{y\in\mathcal{V}:\langle x,Jy\rangle=0\textrm{ for all }x\in S\}$$ is called the {\it symplectic complement} of $S$ in $\mathcal{V}.$ For any given nonempty set $S,$ $S^{\perp_s}$ is always a subspace of $\mathcal{V}.$ If $S$ is a subspace of $\mathcal{V},$ then $\dim\, S+\dim\, S^{\perp_s}=\dim\mathcal{V}.$ A subspace $\mathcal{U}$ of a symplectic space $\mathcal{V}$ is symplectic if and only if $\mathcal{U}\cap\mathcal{U}^{\perp_s}=\{0\}.$ In this case, $\mathcal{V}$ is the direct sum of $\mathcal{U}$ and its symplectic complement. Two vectors $x$ and $y$ are said to be {\it skew-orthogonal} if their symplectic inner product $\langle x,Jy\rangle=0.$ A nonempty set is called {\it skew-orthogonal} if all its vectors are mutually skew-orthogonal. A skew-orthogonal subspace of $\mathbb{R}^{2n}$ is called {\it isotropic}. A subspace $\mathcal{V}$ is isotropic if and only if $\mathcal{V}^{\perp_s}\supseteq\mathcal{V}.$ Clearly, the dimension of an isotropic subspace of a $2m$-dimensional symplectic space is always less than or equal to $m.$ A pair of vectors $(x,y)$ is said to be {\it normalised symplectic pair} if $\langle x,Jy\rangle=1,$ and a set $\{x_1,\ldots,x_k,y_1,\ldots,y_k\}$ is said to be {\it symplectically orthonormal} if $$\langle x_i,Jx_j\rangle =0=\langle y_i,Jy_j\rangle$$ and $$\langle x_i,Jy_j\rangle =\delta_{ij}$$ for all $i,j=1,\ldots,k.$ Every symplectically orthonormal set is linearly independent. If $\mathcal{V}$ is a $2k$-dimensional symplectic space and $\{x_1,\ldots,$ $x_k,y_1,\ldots,y_k\}$ is a symplectically orthonormal subset of $\mathcal{V},$ we call it to be a {\it symplectic basis} of $\mathcal{V}.$ The columns of a $2n\times 2n$ symplectic matrix form a symplectic basis of $\mathbb{R}^{2n}.$ A basis of a symplectic space $\mathcal{V}$ that is symplectic as well as orthonormal is called an {\it orthosymplectic basis} of $\mathcal{V}.$ The standard basis $\{e_1,\ldots,e_{2n}\}$ of $\mathbb{R}^{2n}$ is an orthosymplectic basis of $\mathbb{R}^{2n}.$ For more details on symplectic spaces we refer the reader to \cite{dms, dg}. \vskip.1in Let $\mathcal{V}$ be a symplectic space and let $\mathcal{B}=\{u_1,\ldots,u_m,v_1,\ldots,v_m\}$ be a symplectic basis of $\mathcal{V}.$ We define a new inner product $\langle\cdot,\cdot\rangle_\mathcal{B}$ on $\mathcal{V}$ as follows, (see \cite{bj2}): if $x$ and $y$ are two vectors of $\mathcal{V},$ we can write $x$ and $y$ as $$x=\sum\limits_{i=1}^{m}\left(\alpha_iu_i+\beta_iv_i\right)\textrm{ and }y=\sum\limits_{i=1}^{m}\left(\gamma_iu_i+\delta_iv_i\right).$$ Then $$\langle x,y\rangle_\mathcal{B}=\sum\limits_{i=1}^{m}\left(\alpha_i\gamma_i+\beta_i\delta_i\right).$$ We denote the norm of $x$ in this inner product by $\|x\|_\mathcal{B}.$ Clearly, $\mathcal{B}$ is an orthosymplectic basis of $\mathcal{V}$ under the inner product $\langle\cdot,\cdot\rangle_\mathcal{B}$ and the symplectic inner product $(\cdot,\cdot).$ We call this {\it $\mathcal{B}$-orthosymplectic basis} of $\mathcal{V}.$ We define the {\it $\mathcal{B}$-complement} of $x$ to be the vector $$x^\prime=\sum\limits_{i=1}^{m}\left(-\beta_iu_i+\alpha_iv_i\right).$$ One can see that $x^{\prime\prime}=-x$ for every $x\in\mathcal{V},$ and $u_j^\prime=v_j$ and $v_j^\prime=-u_j,$ $1\le j\le m.$ Also \begin{equation}\label{eq3} \langle x,y\rangle_\mathcal{B}=\langle x,Jy^\prime\rangle=\langle y,Jx^\prime\rangle=\langle x^\prime,y^\prime\rangle_\mathcal{B}. \end{equation} So, if $\{x_1,\ldots,x_k\}$ is a skew-orthogonal, $\mathcal{B}$-orthonormal subset of $\mathcal{V},$ then $\{x_1,\ldots,x_k,x_1^\prime,\ldots,x_k^\prime\}$ is a $\mathcal{B}$-orthosymplectic (symplectic as well as $\mathcal{B}$-orthonormal) set. \vskip.1in For a subspace $\mathcal{W}$ of $\mathcal{V},$ the {\it $\mathcal{B}$-complement} of $\mathcal{W}$ is the set $$\mathcal{W}^\prime=\{x^\prime:x\in\mathcal{W}\}.$$ Clearly, $\mathcal{W}^\prime$ is also a subspace of $\mathcal{W}$ and $\dim\mathcal{W}^\prime=\dim\mathcal{W}.$ We denote the intersection $\mathcal{W}\cap\mathcal{W}^\prime$ by $\mathcal{W}^\sharp.$ It can be seen that $$\mathcal{W}^\sharp=\{x\in\mathcal{W}:x^\prime\in\mathcal{W}\},$$ and $\dim\mathcal{W}>\frac{1}{2}\dim\mathcal{V}$ implies $\dim\mathcal{W}^\sharp>0.$

Let $A$ be a $2n\times 2n$ real positive definite matrix and let $d$ be a symplectic eigenvalue of $A.$
We call a pair of vectors $(u,v)$ to be a {\it symplectic eigenvector pair} of $A$ corresponding to $d$ if
$$Au=dJv\textrm{ and }Av=-dJu.$$
If $d_1(A),\ldots,d_n(A)$ are the symplectic eigenvalues of $A,$
then the set $\mathcal{B}=\{u_1,\ldots,u_n,$
$v_1,\ldots,v_n\}$ is called a {\it symplectic eigenbasis} of $A$ if
$\mathcal{B}$ is a symplectic basis of $\mathbb{R}^{2n}$ and each $(u_i,v_i)$ is a normalised symplectic eigenvector pair of $A$ corresponding to $d_i(A).$
The columns of the symplectic matrix $M$ in \eqref{eqi1} form a symplectic eigenbasis of $\mathbb{R}^{2n}$ corresponding to the symplectic eigenvalues $d_1(A),\ldots,d_n(A)$ of $A.$
With this terminology, we can alternatively state Williamson's theorem as follows:
Let $A$ be a $2n\times 2n$ real positive definite matrix. Then there exist positive numbers $d_1\le\cdots\le d_n$
and a symplectic basis $\mathcal{B}=\{u_1,\ldots,u_n,v_1,\ldots,v_n\}$ of $\mathbb{R}^{2n}$
such that $d_1,\ldots,d_n$ are the symplectic eigenvalues of $A$ and $\mathcal{B}$ is a corresponding symplectic eigenbasis.

\section{Proof of Theorem \ref{thm_wndt}}

We start with an inequality on symplectic eigenvalues that is analogous to the Poincar\'{e} inequality for eigenvalues. See Theorem III.1.1 of \cite{rbh}.

\begin{prop}\label{p1}
Let $A$ be a $2n\times 2n$ real positive definite matrix with symplectic eigenvalues $d_1\le\cdots\le d_n.$
Let $1\le k\le n.$ Then for any $2n-k+1$-dimensional subspace $\mathcal{M},$
there exist two vectors $u,v$ in $\mathcal{M}$ such that $\langle u,Jv\rangle=1$ and
\begin{equation}\label{eqp1}
d_k\ge\frac{\langle u,Au\rangle+\langle v,Av\rangle}{2}.
\end{equation}
\end{prop}

\begin{proof}
Let $\mathcal{B}=\{u_1,\ldots,u_n,v_1,\ldots,v_n\}$ be a symplectic eigenbasis of $\mathbb{R}^{2n}$ corresponding to
the symplectic eigenvalues $d_1,\ldots,d_n$ of $A.$
Let $\mathcal{N}$ be the space spanned by $\{u_1,\ldots,u_n,$
$v_1,\ldots,v_k\}.$
Since $\dim\mathcal{M}+\dim\mathcal{N}=3n+1,$
$\dim\left(\mathcal{M}\cap\mathcal{N}\right)\ge n+1.$
Hence we can find a vector $u$ in $\mathcal{M}\cap\mathcal{N}$ such that $\|u\|_\mathcal{B}=1$
and the $\mathcal{B}$-complement $u^\prime$ of $u$ also belongs to $\mathcal{M}\cap\mathcal{N}.$
Since $u,u^\prime$ both belong to $\mathcal{N}$ and $\mathcal{N}$ is spanned by
$u_1,\ldots,u_n,$
$v_1,\ldots,v_k,$
we must have
$$u=\sum\limits_{i=1}^{k}\left(\alpha_iu_i+\beta_iv_i\right),$$
where $\sum_{i=1}^{k}\left(\alpha_i^2+\beta_i^2\right)=1.$
We thus have $$u^\prime=\sum\limits_{i=1}^{k}\left(-\beta_iu_i+\alpha_iv_i\right).$$
A straightforward calculation shows that
\begin{equation}
\langle u,Au\rangle=\sum\limits_{i=1}^{k}\left(\alpha_i^2+\beta_i^2\right)d_i\le d_k.\label{eqrp1}
\end{equation}
Similarly, $\langle u^\prime,Au^\prime\rangle\le d_k.$
Since $\langle u,Ju^\prime\rangle=1,$
$u$ and $u^\prime$ are the required vectors that satisfy \eqref{eqp1}.
\end{proof}

We next prove a maxmin principle for symplectic eigenvalues. Two versions of this principle can be found in \cite{dg,bj} and \cite{bj2}.
Though our present version is similar to that given in Theorem 1 of \cite{bj2},
the maxmin principle in its present form plays a key role in proving Theorem \ref{thm_wndt}.
  
\begin{thm}\label{thm1maxmin}
Let $A$ be a $2n\times 2n$ real positive definite matrix,
and let $d_1\le \cdots\le d_n$ be the symplectic eigenvalues of $A.$
Then
\begin{equation}\label{eq1m}
d_k={\underset{\underset{dim\mathcal{M}=2n-k+1}{\mathcal{M}\subseteq\mathbb{R}^{2n}}}{\max}}{\underset{\underset{\langle x,Jy\rangle=1}{x,y\in\mathcal{M}}}{\min}}\frac{\langle x,Ax\rangle+\langle y,Ay\rangle}{2},\ \ 1\le k\le n.
\end{equation}
\end{thm}

\begin{proof}
Let $\mathcal{M}$ be a $2n-k+1$-dimensional subspace of $\mathbb{R}^{2n}.$
By Proposition \ref{p1}
we can find a normalised symplectic pair $(u,v)$ that satisfy \eqref{eqp1}.
Then
\begin{eqnarray*}
d_k &\ge & \frac{\langle u,Au\rangle+\langle v,Av\rangle}{2}\\
&\ge & {\underset{\underset{\langle x,Jy\rangle=1}{x,y\in\mathcal{M}}}{\min}}\frac{\langle x,Ax\rangle+\langle y,Ay\rangle}{2}.
\end{eqnarray*}
This shows that the right hand side of \eqref{eq1m} is less than or equal to $d_k.$
\vskip.1in
Let $\mathcal{B}=\{u_1,\ldots,u_n,v_1,\ldots,v_n\}$ be a symplectic eigenbasis of $\mathbb{R}^{2n}$ corresponding to the symplectic eigenvalues $d_1,\ldots,d_n$ of $A.$
Let $\mathcal{M}_k$ be the subspace spanned by $\{u_1,\ldots,u_n,$
$v_k,\ldots,v_n\}.$
Clearly $\dim\mathcal{M}_k=2n-k+1.$
Let $x=\sum_{i=1}^{n}\alpha_iu_i+\sum_{i=k}^{n}\beta_iv_i$ and $y=\sum_{i=1}^{n}\gamma_iu_i+\sum_{i=k}^{n}\delta_iv_i$
be any two elements of $\mathcal{M}_k$ such that $\langle x,Jy\rangle=1.$
We have
\begin{eqnarray*}
1 &=& \langle x,Jy\rangle\\
&=& \sum\limits_{i=k}^{n}\left(\alpha_i\delta_i-\beta_i\gamma_i\right)\\
&\le & \sum\limits_{i=k}^{n}\frac{\alpha_i^2+\beta_i^2+\gamma_i^2+\delta_i^2}{2}.
\end{eqnarray*}
The last inequality follows from the arithmetic-geometric mean inequality.
Now we have
\begin{eqnarray}
 \frac{\langle x,Ax\rangle+\langle y,Ay\rangle}{2}\ & =&\sum\limits_{i=1}^{n}\frac{\alpha_i^2+\gamma_i^2}{2}d_i+\sum\limits_{i=k}^{n}\frac{\beta_i^2+\delta_i^2}{2}d_i \nonumber\\
 &\ge &d_k\sum\limits_{i=k}^{n}\frac{\alpha_i^2+\beta_i^2+\gamma_i^2+\delta_i^2}{2} \nonumber\\
 &\ge& d_k .\label{eqs1}
\end{eqnarray}
Thus
\begin{equation}
d_k\le {\underset{\underset{\langle x,Jy\rangle=1}{x,y\in\mathcal{M}_k}}{\min}}\frac{\langle x,Ax\rangle+\langle y,Ay\rangle}{2}.\label{eq2m}
\end{equation}
The two sides of \eqref{eq2m} are equal for $x=u_k$ and $y=v_k.$
This gives the equality \eqref{eq1m}.
\end{proof}

\begin{lem}\label{lemr1}
Let $\mathcal{V}$ be a $2n$-dimensional symplectic space and let $\mathcal{B}$ be a symplectic basis of $\mathcal{V}.$
For $1\le k\le n,$ let $\mathcal{W}_1\supseteq\cdots\supseteq\mathcal{W}_k$ be a decreasing chain of subspaces of $\mathcal{V}$
with $\dim\mathcal{W}_j\ge n+k-j+1,$ $1\le j\le k.$
Let $\{w_1,\ldots,w_{k-1}\}$ be a $\mathcal{B}$-orthonormal, skew-orthogonal set
such that $w_j\in\mathcal{W}_j^\sharp,$
$j=1,\ldots,k-1,$
and let $\mathcal{U}$ be the space spanned by $\{w_1,\ldots,w_{k-1},w_1^\prime,\ldots,w_{k-1}^\prime\}.$
Then there exists an element $v\in\mathcal{W}_1^\sharp\cap\{w_1,\ldots,w_{k-1}\}^{\perp_s}$
such that $\mathcal{U}+\span\{v,v^\prime\}$ has a $\mathcal{B}$-orthosymplectic basis
$\{v_1,\ldots,v_k,v_1^\prime,\ldots,v_k^\prime\},$
where $v_j\in\mathcal{W}_j^\sharp$ for all  $j=1,\ldots,k.$
\end{lem}

\begin{proof}
Since $\dim\mathcal{W}_j\ge n+(k-j+1),$
$\dim\mathcal{W}_j^\sharp\ge 2(k-j+1),$ $1\le j\le k.$
We prove the result by induction on $k.$
Let $k=2,$ and let $w_1\in\mathcal{W}_1^\sharp.$
Then $\mathcal{U}=\span\{w_1,w_1^\prime\}.$
We know that $\dim\mathcal{W}_1^\sharp\ge 4$ and $\dim\, \mathcal{U}^{\perp_s}=2n-2.$
So, if $w_1\in\mathcal{W}_2^\sharp,$ then there exists a $v\in\mathcal{W}_1^\sharp\cap\mathcal{U}^{\perp_s}.$
The set $\{v,w_1,v^\prime,w_1^\prime\}$ is the required $\mathcal{B}$-orthosymplectic basis of $\mathcal{U}+\span\{v,v^\prime\}.$
Now, suppose that $w_1\notin\mathcal{W}_2^\sharp.$
Since $\dim\mathcal{W}_2^\sharp\ge 2,$
we can find a $v\in\mathcal{W}_2^\sharp\cap\{w_1\}^{\perp_s}.$
By the classical Gram-Schmidt orthogonalisation on $\{v,w_1\}$ in the space $(\mathcal{V},\langle\cdot,\cdot\rangle_\mathcal{B}),$
we can find a $\mathcal{B}$-orthonormal, skew-orthogonal elements $v_2,v_1$ in $\mathcal{W}_2$ and $\mathcal{W}_1,$ respectively,
such that $\span\{w_1,v\}=\span\{v_1,v_2\}.$
Hence $\{v_1,v_2,v_1^\prime,v_2^\prime\}$ forms the required $\mathcal{B}$-orthosymplectic basis of $\mathcal{U}+\span\{v,v^\prime\}.$

Assume that the result holds for $k-1.$
We prove it for $k.$
By the induction hypotheses,
there exists a vector $u$ in $\mathcal{W}_2^\sharp\cap\{w_2,\ldots,w_{k-1}\}^{\perp_s}$ such that
$$\mathcal{S}=\span\{w_2,\ldots,w_{k-1},u,w_2^\prime,\ldots,w_{k-1}^\prime,u^\prime\}=\span\{x_2,\ldots,x_k,x_2^\prime,\ldots,x_k^\prime\}$$
for some $\mathcal{B}$-orthonormal, skew-orthogonal $x_j\in\mathcal{W}_j^\sharp,$ $2\le j\le k.$
Suppose $\mathcal{U}$ is the symplectic space spanned by $w_1,\ldots,w_{k-1},w_1^\prime,\ldots,w_{k-1}^\prime.$
Two cases arise depending on $u\in\mathcal{U}$ or $u\notin\mathcal{U}.$
Consider the case $u\in\mathcal{U}.$
Since $\mathcal{S}\subseteq\mathcal{U}$ and $\dim\mathcal{S}=\dim\, \mathcal{U}=2(k-1),$
$\mathcal{S}=\mathcal{U}.$
We know that $\dim\mathcal{W}_1^\sharp\ge 2k$ and $\dim\mathcal{U}^{\perp_s}=2(n-k+1).$
Hence there exists a vector $v$ in $\mathcal{W}_1^\sharp\cap\mathcal{U}^{\perp_s}$ with $\|v\|_\mathcal{B}=1.$
Since $v$ is skew-orthogonal to each element of $\mathcal{U},$
it is $\mathcal{B}$-orthogonal and skew-orthogonal to all $x_j$'s.
Take $v_1=v$ and $v_j=x_j,$ $2\le j\le k.$
Then $\{v_1,\ldots,v_k,v_1^\prime,\ldots,v_k^\prime\}$ is the required $\mathcal{B}$-orthosymplectic basis of $\mathcal{U}+\span\{v,v^\prime\}.$

Now consider the case $u\notin\mathcal{U}.$
If $u\in\mathcal{U}^{\perp_s},$
then take $v=u.$
As in the preceding paragraph,
we see that $\{v,x_2,\ldots,x_k,v^\prime,x_2^\prime,\ldots,x_k^\prime\}$ forms the required $\mathcal{B}$-orthosymplectic basis of $\mathcal{U}+\span\{v,v^\prime\}.$
So, let $u\notin\mathcal{U}^{\perp_s}.$
We can write $u=u_1+u_2,$ where $u_1\in\mathcal{U}$ and $u_2\in\mathcal{U}^{\perp_s}.$
Since $u,u_1\in\mathcal{W}_1^\sharp,$ $u_2$ also belongs to $\mathcal{W}_1^\sharp.$
Take $v=u_2,$ and consider the space $\mathcal{U}_0=\mathcal{U}+\span\{v,v^\prime\}.$
Clearly, $u\in\mathcal{U}_0$ and $\{x_2,\ldots,x_k\}$ is a $\mathcal{B}$-orthonormal, skew-orthogonal set in $\mathcal{U}_0.$
Since $\mathcal{U}_0$ is a symplectic space and $\dim\mathcal{U}_0=2k>\dim\mathcal{S},$
we can find a vector $v_1$ in $\mathcal{U}_0$ with $\|v_1\|_\mathcal{B}=1$
that is $\mathcal{B}$-orthogonal and skew-orthogonal to $x_2,\ldots,x_k.$
Taking $v_j=x_j$ for $2\le j\le k,$
we can see that 
$\{v_1,\ldots,v_k,v_1^\prime,\ldots,v_k^\prime\}$ is the required $\mathcal{B}$-orthosymplectic basis of $\mathcal{U}_0.$
\end{proof}

\begin{thm}\label{thmr2}
Let $\mathcal{B}$ be a symplectic basis of $\mathbb{R}^{2n}.$
Let $\mathcal{V}_1\subset\cdots\subset\mathcal{V}_k$ be an increasing chain of subspaces,
and let $\mathcal{W}_1\supset\cdots\supset\mathcal{W}_k$ be a decreasing chain of subspaces of $\mathbb{R}^{2n}$
such that $\dim\mathcal{V}_j=n+i_j$ and $\dim\mathcal{W}_j=2n-i_j+1$ for $j=1,\ldots,k,$
and $1\le i_1<\cdots < i_k\le n.$
Then we can find two $\mathcal{B}$-orthosymplectic sets
$\{v_1,\ldots,v_k,v_1^\prime,\ldots,v_k^\prime\}$ and $\{w_1,\ldots,w_k,w_1^\prime,\ldots,w_k^\prime\}$
that have the same span and are such that
$v_j\in\mathcal{V}_j^\sharp$ and $w_j\in\mathcal{W}_j^\sharp$ for all $j=1,\ldots,k.$
\end{thm}

\begin{proof}
We first note that for every $j=1,\ldots,k,$ $\dim\mathcal{V}_j^\sharp\ge 2i_j$ and $\dim\mathcal{W}_j^\sharp\ge 2(n-i_j+1).$
We use induction on $k$ to prove the theorem.
When $k=1,$ there exists an element $x\in\mathcal{V}_1^\sharp\cap\mathcal{W}_1^\sharp$ with $\|x\|_\mathcal{B}=1.$
Take $v_1=w_1=x$ to obtain the result for $k=1.$

\sloppy
Assume that the result holds for $k-1.$
Then there exist two $\mathcal{B}$-orthosymplectic sets $\{v_1,\ldots,v_{k-1},v_1^\prime,\ldots,v_{k-1}^\prime\}$
and $\{x_1,\ldots,x_{k-1},x_1^\prime,\ldots,x_{k-1}^\prime\}$
that have the same span $\mathcal{U}$ and are such that $v_j\in\mathcal{V}_j^\sharp$ and $x_j\in\mathcal{W}_j^\sharp$ for all $j=1,\ldots,k-1.$
We now show that there exist vectors $v_k\in\mathcal{V}_k^\sharp$ and $w_j\in\mathcal{W}_j^\sharp,$ $1\le j\le k,$
such that the vectors $v_1,\ldots,v_k$ and $w_1,\ldots,w_k$ satisfy the required properties.

\sloppy
For $j=1,\ldots,k$ consider the subspace $\mathcal{S}_j=\mathcal{V}_k\cap\mathcal{W}_j.$
Clearly $\dim\mathcal{S}_j\ge n+i_k-i_j+1\ge n+k-j+1.$
Also for all $j=1,\ldots,k-1,$ $x_j\in\mathcal{S}_j.$
By Lemma \ref{lemr1}
we can find a $v\in\mathcal{S}_1^\sharp\cap\{x_1,\ldots,x_{k-1}\}^{\perp_s}$
such that the space spanned by $\{x_1,\ldots,x_{k-1},v,x_1^\prime,\ldots,x_{k-1}^\prime,v^\prime\}$
has a $\mathcal{B}$-orthosymplectic basis
$\{w_1,\ldots,w_k,w_1^\prime,\ldots,w_k^\prime\},$ where $w_j\in\mathcal{S}_j^\sharp$ for all $j=1,\ldots,k.$
By the classical Gram-Schmidt orthogonalisation on $x_1,\ldots,x_{k-1},v$
in the space $(\mathbb{R}^{2n},\langle\cdot,\cdot\rangle_\mathcal{B}),$
we get a $\mathcal{B}$-unit vector $v_k\in\mathcal{S}_1^\sharp\subseteq\mathcal{V}_k^\sharp$
that is $\mathcal{B}$-orthogonal to all $x_1,\ldots,x_{k-1},$
and is such that $\span\{x_1,\ldots,x_{k-1},v\}=\span\{x_1,\ldots,x_{k-1},v_k\}.$
Since $\{x_1,\ldots,x_{k-1},v\}$ is a skew-orthogonal set, $v_k$ is also skew-orthogonal to all $x_j$'s, and consequently to the space $\mathcal{U}.$
Now the set $\{v_1,\ldots,v_{k-1},v_1^\prime,\ldots,v_{k-1}^\prime\}$ is a $\mathcal{B}$-orthosymplectic basis of $\mathcal{U}.$
Thus $\{v_1,\ldots,v_k,v_1^\prime,\ldots,v_k^\prime\}$ is a $\mathcal{B}$-orthosymplectic set that has the same span as $\{w_1,\ldots,w_k,w_1^\prime,\ldots,w_k^\prime\}.$
This proves the theorem.
\end{proof}

\begin{prop}\label{propr3}
Let $\mathcal{B}$ be a symplectic eigenbasis of $\mathbb{R}^{2n}$
corresponding to the symplectic eigenvalues $d_1\le\cdots\le d_n$ of a $2n\times 2n$ real positive definite matrix $A.$
Let $\{x_1,\ldots,x_k,x_1^\prime,\ldots,x_k^\prime\}$
and $\{v_1,\ldots,v_k,v_1^\prime,\ldots,v_k^\prime\}$
be two $\mathcal{B}$-orthosymplectic sets in $\mathbb{R}^{2n}$that have the same span.
 Then
$$\sum\limits_{j=1}^{k}\left(\langle x_j,Ax_j\rangle+\langle x_j^\prime,Ax_j^\prime\rangle\right)=\sum\limits_{j=1}^{k}\left(\langle v_j,Av_j\rangle+\langle v_j^\prime,Av_j^\prime\rangle\right).$$
\end{prop}

\begin{proof}
Let $\mathcal{B}=\{w_1,\ldots,w_n,z_1,\ldots,z_n\}.$
We can write each vector $x$ in $\mathbb{R}^{2n}$ as
$$x=\sum\limits_{i=1}^{n}\left(\alpha_iw_i+\beta_iz_i\right).$$
Then
\begin{equation}
\langle x,Ax\rangle=\sum\limits_{i=1}^{n}d_i\left(\alpha_i^2+\beta_i^2\right).\label{eq31}
\end{equation}
Let $\wt{D}$ be the operator on $\mathbb{R}^{2n}$ defined as
$$\wt{D}w_i=d_iw_i\textrm{ and }\wt{D}z_i=d_iz_i,$$
$1\le i\le n.$
Clearly, $\wt{D}$ is a diagonal operator on $(\mathbb{R}^{2n},\langle\cdot,\cdot\rangle_\mathcal{B}).$
We have $\wt{D}x=\sum\limits_{i=1}^{n}d_i\left(\alpha_iw_i+\beta_iz_i\right).$
By using the definition of $\langle\cdot,\cdot\rangle_\mathcal{B}$ and equality \eqref{eq31},
we get
\begin{equation}
\langle x,\wt{D}x\rangle_\mathcal{B}=\sum\limits_{i=1}^{n}d_i\left(\alpha_i^2+\beta_i^2\right)=\langle x,Ax\rangle.\label{eq32}
\end{equation}
Since $\{x_1,\ldots,x_k,x_1^\prime,\ldots,x_k^\prime\}$ and $\{v_1,\ldots,v_k,v_1^\prime,\ldots,v_k^\prime\}$ are $\mathcal{B}$-orthonormal and have the same span, we have
\begin{equation}
\sum\limits_{i=1}^{k}\left(\langle x_i,\wt{D}x_i\rangle_\mathcal{B}+\langle x_i^\prime,\wt{D}x_i^\prime\rangle_\mathcal{B}\right)=\sum\limits_{i=1}^{k}\left(\langle v_i,\wt{D}v_i\rangle_\mathcal{B}+\langle v_i^\prime,\wt{D}v_i^\prime\rangle_\mathcal{B}\right).\label{eq33}
\end{equation}
Thus by using the relations \eqref{eq32} and \eqref{eq33}, we obtain the proposition.
\end{proof}
\vskip.1in

\noindent{\bf Proof of Theorem \ref{thm_wndt}}:
Let $\mathcal{B}=\{u_1,\ldots,u_n,v_1,\ldots,v_n\}$ be a symplectic eigenbasis of $\mathbb{R}^{2n}$ corresponding to the symplectic eigenvalues $d_1,\ldots,d_n$ of $A.$
Let $\mathcal{M}_j$ be the $2n-i_j+1$-dimensional subspace spanned by
$\{u_1,\ldots,u_n,$
$v_{i_j},\ldots,v_n\}.$
Let $1\le j\le k,$ and let $(x,y)$ be any normalised symplectic pair of vectors in $\mathcal{M}_j.$
By \eqref{eqs1}, we see that
$$d_{i_j}\le \frac{\langle x,Ax\rangle+\langle y,Ay\rangle}{2}.$$
Hence, for any symplectically orthonormal set
$\{x_1,\ldots,x_k,y_1,\ldots,y_k\}$ with $x_j,y_j\in\mathcal{M}_j,$
we have
$$\sum\limits_{j=1}^{k}d_{i_j}\le \sum\limits_{j=1}^{k}\frac{\langle x_j,Ax_j\rangle+\langle y_j,Ay_j\rangle}{2}.$$
The two sides are equal when we take $x_j=u_{i_j}$ and $y_j=v_{i_j},$ $1\le j\le k.$
Thus
$$\sum\limits_{j=1}^{k}d_{i_j}={\underset{\underset{\underset{\textrm{symp. o. n.}}{\{x_1,\ldots,x_k,y_1,\ldots,y_k\}}}{x_j,y_j\in\mathcal{M}_j}}{\min}}\sum\limits_{j=1}^{k}\frac{\langle x_j,Ax_j\rangle+\langle y_j,Ay_j\rangle}{2}.$$
This shows that the left hand side of \eqref{eqwndt} is less than or equal to its right hand side.

Now let $\mathcal{W}_1\supset\cdots\supset\mathcal{W}_k$ be a decreasing chain of subspaces with
$\dim\mathcal{W}_j=2n-i_j+1.$
Let $\mathcal{V}_j$ be the subspaces spanned by
$\{u_1,\ldots,u_n,$
$v_1,\ldots,$$v_{i_j}\}.$
Clearly $\dim\mathcal{V}_j=n+i_j.$
By Theorem \ref{thmr2}
we can find two $\mathcal{B}$-orthosymplectic sets
$\{x_1,\ldots,x_k,$$x_1^\prime,\ldots,x_k^\prime\}$ and $\{w_1,\ldots,w_k,$$w_1^\prime,\ldots,w_k^\prime\}$
that have the same span $\mathcal{U}$ and are such that $x_j,x_j^\prime\in\mathcal{W}_j$
and $w_j,w_j^\prime\in\mathcal{V}_j$ for all $j=1,\ldots,k.$
Since both $w_j,w_j^\prime$ belong to $\mathcal{V}_j$ and $\mathcal{V}_j$ is spanned by
$\{u_1,\ldots,u_n,v_1,\ldots,v_{i_j}\},$
we must have
$$w_j=\sum\limits_{m=1}^{i_j}\left(\alpha_{jm}u_m+\beta_{jm}v_m\right).$$
Then by using the same argument as in \eqref{eqrp1} and by using Proposition \ref{propr3} we have
\begin{eqnarray*}
\sum\limits_{j=1}^{k}d_{i_j}&\ge& \sum\limits_{j=1}^{k}\frac{\langle w_j,Aw_j\rangle+\langle w_j^\prime,Aw_j^\prime\rangle}{2}\\
&=&\sum\limits_{j=1}^{k}\frac{\langle x_j,Ax_j\rangle+\langle x_j^\prime,Ax_j^\prime\rangle}{2}.
\end{eqnarray*}
This gives
$$\sum\limits_{j=1}^{k}d_{i_j}\ge{\underset{\underset{\underset{\textrm{symp. o. n.}}{\{x_1,\ldots,x_k,y_1,\ldots,y_k\}}}{x_j,y_j\in\mathcal{W}_j}}{\min}}\sum\limits_{j=1}^{k}\frac{\langle x_j,Ax_j\rangle+\langle y_j,Ay_j\rangle}{2}.$$
 This proves equality \eqref{eqwndt}.
\qed
\vskip.1in

%We obtain a symplectic analogue of Lidskii's inequalities as a corollary to Theorem \ref{thm_wndt}.This was proved in \cite{jm} by %using very different arguments.

\vskip.1in

\noindent{\bf Proof of Corollary \ref{cor_lid}}:
Let $\mathcal{W}_1\supset\cdots\supset\mathcal{W}_k$ be a decreasing chain of subspaces with $\dim\mathcal{W}_j=2n-i_j+1$
such that
\begin{equation}
\sum\limits_{j=1}^{k}d_{i_j}(A)={\underset{\underset{\underset{\textrm{symp. o. n.}}{\{u_1,\ldots,u_k,v_1,\ldots,v_k\}}}{u_j,v_j\in\mathcal{W}_j}}{\min}}\sum\limits_{j=1}^{k}\frac{\langle u_j,Au_j\rangle+\langle v_j,Av_j\rangle}{2}.\label{eqlid1}
\end{equation}
By Theorem 5 of \cite{bj}, we have
\begin{equation}
\sum\limits_{j=1}^{k}d_j(B)\le \sum\limits_{j=1}^{k}\frac{\langle u_j,Bu_j\rangle+\langle v_j,Bv_j\rangle}{2}\label{eqlid2}
\end{equation}
for all symplectically orthonormal vectors $u_1,\ldots,u_k,v_1,\ldots,v_k.$
Now, by Theorem \ref{thm_wndt} and relations \eqref{eqlid1} and \eqref{eqlid2}, we obtain
\begin{eqnarray*}
\sum\limits_{j=1}^{k}d_{i_j}(A+B) &\ge & {\underset{\underset{\underset{\textrm{symp. o. n.}}{\{u_1,\ldots,u_k,v_1,\ldots,v_k\}}}{u_j,v_j\in\mathcal{W}_j}}{\min}}\frac{\langle u_j,(A+B)u_j\rangle+\langle v_j,(A+B)v_j\rangle}{2}\\
&\ge & \sum\limits_{j=1}^{k}d_{i_j}(A)+\sum\limits_{j=1}^{k}d_j(B).
\end{eqnarray*}
\qed

\section{Proof of Theorem \ref{thmmlid}}

Let $\alpha=(\alpha_1,\ldots,\alpha_n)$ be a real vector. The vector $\alpha^\uparrow$ denotes the vector $(\alpha_1^\uparrow,\ldots,\alpha_n^\uparrow)$ obtained by arranging the components of $x$ in increasing order.
For any two vectors $\alpha,\beta,$ $\alpha\le \beta$ if $\alpha_i^\uparrow\le \beta_i^\uparrow $ for all $i=1,\ldots,n.$
We say $\alpha$ is {\it supermajorised} by $\beta,$ in symbols $\alpha\prec^w \beta,$ if for all $k=1,\ldots,n,$ 
\begin{equation}
 \sum_{j=1}^{k} \, \alpha_j^{\uparrow} \ge \sum_{j=1}^{k} \,\beta_j^{\uparrow}.\label{eqmaj} 
\end{equation}
 Further $\alpha$ is {\it majorised} by $\beta$ (or $\beta$ {\it majorises} $\alpha$)
if the two sides in \eqref{eqmaj} are equal when $k=n.$

A function $\varphi:\mathbb{R}^n\to\mathbb{R}$ is called {\it Schur concave} if for every $\alpha,\beta$ in $\mathbb{R}^n,$
$$\alpha\prec \beta\, \implies\, \varphi(\alpha)\ge\varphi(\beta).$$
The function $\varphi$ is called monotonically increasing if $\alpha\le \beta$ implies $\varphi(\alpha)\le\varphi(\beta).$
We refer the reader to \cite[Theorem 8.8, p.87]{mo} for the following lemma.

\begin{lem}\label{lemwg1}
Let $\varphi:\mathbb{R}^n_+\to\mathbb{R}$ be a Schur-concave and monotonically increasing function.
Then for any two vectors $x,$ $y$ in $\mathbb{R}^n_+$ with $x\prec^w y,$
we have $\varphi(x)\ge \varphi(y).$
\end{lem}

We now give a generalisation of Theorem \ref{thm_wndt}.

\begin{thm}\label{thmwg}
Let $1\le k\le n,$ and let $\varphi:[0,\infty)^k\to\mathbb{R}_+$ be a Schur-concave, permutation invariant and a monotonically increasing function.
Let $A$ be a $2n\times 2n$ real positive definite matrix with symplectic eigenvalues $d_1\le \cdots\le d_n.$
Then for every $1\le i_1<\cdots<i_k\le n,$
\begin{equation}
\varphi((d_{i_1},\ldots,d_{i_k}))={\underset{\underset{\dim\mathcal{M}_j=2n-i_j+1}{\mathcal{M}_1\supset\cdots\supset\mathcal{M}_k}}{\max}}{\underset{\underset{\underset{\textrm{symp. o. n.}}{u_j,v_j\in\mathcal{M}_j}}{\mathcal{M}=\span\{u_1,\ldots,u_k,v_1,\ldots,v_k\}}}{\min}}\varphi(d_\mathcal{M}).\label{eqwg1}
\end{equation}
Here $d_\mathcal{M}$ denotes the $k$-vector of symplectic eigenvalues of the real positive definite matrix $A_\mathcal{M}$ obtained by restricting $A$ to $\mathcal{M}.$
\end{thm}

\begin{proof}
Let $\mathcal{B}=\{w_1,\ldots,w_n,z_1,\ldots,z_n\}$
be a symplectic eigenbasis of $\mathbb{R}^{2n}$ corresponding to the symplectic eigenvalues $d_1,\ldots,d_n$ of $A.$
For each $j=1,\ldots,k,$ let $\mathcal{W}_j$ be the space spanned by $\{w_1,\ldots,w_n,z_1,\ldots,z_{i_j}\},$ and
let $\mathcal{M}_1\supset\mathcal{M}_2\supset\cdots\supset\mathcal{M}_k$ be a decreasing chain of subspaces of $\mathbb{R}^{2n}$ such that $\dim\mathcal{M}_j=2n-i_j+1.$
By Theorem \ref{thmr2},
we can find a symplectic subspace $\mathcal{U},$
and its two $\mathcal{B}$-orthosymplectic bases
$\{u_1,\ldots,u_k,v_1,\ldots,v_k\}$ and $\{x_1,\ldots,x_k,y_1,\ldots,y_k\}$
such that $u_j,v_j\in\mathcal{M}_j$ and $x_j,y_j\in\mathcal{W}_j$ for all $j=1,\ldots,k.$
Then as in \eqref{eqrp1}, we have
\begin{equation}
d_{i_j}\ge \frac{\langle x_j,Ax_j\rangle+\langle y_j,Ay_j\rangle}{2},\ 1\le j\le k.\label{eqwg2}
\end{equation}
Let $\alpha$ be the $n$-vector $\left(\frac{\langle x_j,Ax_j\rangle+\langle y_j,Ay_j\rangle}{2}\right).$
Let $d_\mathcal{U}$ be the vector of symplectic eigenvalues of the real positive definite matrix $A_\mathcal{U}$ obtained by restricting $A$ to $\mathcal{U}.$
Then by Theorem 3 of \cite{bjsh}, we know that $\alpha\prec^w d_\mathcal{U}.$
By Lemma \ref{lemwg1}, we have $\varphi(\alpha)\ge\varphi(d_\mathcal{U}).$
Since $\alpha\le (d_{i_1},\ldots,d_{i_k})$ and $\varphi$ is increasing,
$\varphi(\alpha)\le\varphi((d_{i_1},\ldots,d_{i_k})).$
This proves that the left hand side of \eqref{eqwg1} is greater than or equal to its right hand side.

\sloppy
To prove the equality,
we consider the subspaces $\mathcal{M}_j$ spanned by the vectors $\{w_1,\ldots,w_n,z_{i_j},\ldots,z_{n}\}.$
Clearly $\mathcal{M}_1\supset\cdots\supset\mathcal{M}_k$ and $\dim\mathcal{M}_j=2n-i_j+1.$
Let $\mathcal{M}$ be the span of any symplectically orthonormal set $\{x_1,\ldots,x_k,y_1,\ldots,y_k\}$
where $x_j,y_j\in\mathcal{M}_j,$ $1\le j\le k.$
For $j=1,\ldots,k,$ let $\mathcal{N}_j$ be the subspace of $\mathcal{M}$ spanned by $\{x_1,\ldots,x_k,y_j,\ldots,y_k\}.$
By \eqref{eq2m}, we can see that
\begin{equation}
d_{i_j}\le {\underset{\underset{\langle x,Jy\rangle=1}{x,y\in\mathcal{M}_j}}{\min}}\frac{\langle x,Ax\rangle+\langle y,Ay\rangle}{2}\le {\underset{\underset{\langle x,Jy\rangle=1}{x,y\in\mathcal{N}_j}}{\min}}\frac{\langle x,Ax\rangle+\langle y,Ay\rangle}{2}.\label{eqs2}
\end{equation}
Let $\tilde{d_1}\le \cdots\le \tilde{d_k}$ be the symplectic eigenvalues of $A_\mathcal{M}.$
By using arguments similar to those in Theorem \ref{thm1maxmin} for $A_\mathcal{M},$
we have
\begin{equation*}
\tilde{d_j}={\underset{\underset{dim\mathcal{N}=2k-j+1}{\mathcal{N}\subseteq\mathcal{M}}}{\max}}{\underset{\underset{\langle x,Jy\rangle=1}{x,y\in\mathcal{N}}}{\min}}\frac{\langle x,Ax\rangle+\langle y,Ay\rangle}{2},\ \ 1\le k\le n.
\end{equation*}
This in turn gives
\begin{equation}
\tilde{d_j}\ge{\underset{\underset{\langle x,Jy\rangle=1}{x,y\in\mathcal{N}_j}}{\min}}\frac{\langle x,Ax\rangle+\langle y,Ay\rangle}{2}.\label{eqs3}
\end{equation}
Combining \eqref{eqs2} and \eqref{eqs3}, we get $d_{i_j}\le \tilde{d_j}$ for all $j=1,\ldots,k.$
Since $\varphi$ is monotonically increasing, $\varphi((d_{i_1},\ldots,d_{i_k}))\le\varphi((\tilde{d_1},\ldots,\tilde{d_k}))=\varphi(d_\mathcal{M}).$
Thus we have
$$\varphi((d_{i_1},\ldots,d_{i_k}))\le{\underset{\underset{\underset{\textrm{symp. o. n.}}{u_j,v_j\in\mathcal{M}_j}}{\mathcal{M}=\span\{u_1,\ldots,u_k,v_1,\ldots,v_k\}}}{\min}}\varphi(d_\mathcal{M}),$$
which proves \eqref{eqwg1}.
\end{proof}

By taking $\varphi((\alpha_1,\ldots,\alpha_k))=\alpha_1\cdots\alpha_k,$
we obtain the following corollary.
This extends Theorem 5(ii) of \cite{bj}.

\begin{cor}\label{corprod}
Let $A$ be a $2n\times 2n$ real positive definite matrix with symplectic eigenvalues $d_1\le\cdots\le d_n.$
Then for all $k=1,\ldots,n$ and $1\le i_1<\cdots<i_k\le n,$
\begin{equation}
\prod\limits_{j=1}^{k}d_{i_j}^2(A)={\underset{\underset{\dim\mathcal{W}_j=2n-i_j+1}{\mathcal{W}_1\supset\cdots\supset\mathcal{W}_k}}{\max}}{\underset{\underset{\underset{\textrm{ symplectically orthonormal}}{u_j,v_j\in\mathcal{W}_j}}{\mathcal{M}=\span\{u_1,\ldots,u_k,v_1,\ldots,v_k\}}}{\min}}\det(A_\mathcal{M}).\label{eqprod}
\end{equation}
\end{cor}
\vskip.1in

\noindent{\bf Proof of Theorem \ref{thmmlid}}:
Let $\mathcal{W}_1\supset\cdots\supset \mathcal{W}_k$ be a decreasing chain of subspaces of $\mathbb{R}^{2n}$ with $\dim\mathcal{W}_j=2n-i_j+1.$
Suppose $u_j,v_j\in\mathcal{W}_j$ are such that $\{u_1,\ldots,u_k,v_1,\ldots,v_k\}$ is symplectically orthonormal.
Let $\mathcal{M}$ be the span of $\{u_1,\ldots,u_k,v_1,\ldots,v_k\}.$
For the convenience of notation, we denote $v_1,\ldots,v_k$ by $u_{k+1},\ldots,u_{2k},$ respectively.
Let $U$ be the orthogonal matrix such that $A\# B=A^{1/2}UB^{1/2}.$
Now,
\begin{eqnarray}	
\left\vert\det(A\# B)_\mathcal{M}\right\vert^2&=&\left\vert\det\begin{bmatrix}\langle u_iA^{1/2}UB^{1/2}u_j\rangle\end{bmatrix}\right\vert^2\nonumber\\	
&=&\left\vert\det\begin{bmatrix}\langle A^{1/2}u_i,UB^{1/2}u_j\rangle\end{bmatrix}\right\vert^2\nonumber\\
&\le & \left\vert\det\begin{bmatrix}\langle A^{1/2}u_i,A^{1/2}u_j\rangle\end{bmatrix}\right\vert\left\vert\det\begin{bmatrix}\langle UB^{1/2}u_i,UB^{1/2}u_j\rangle\end{bmatrix}\right\vert\nonumber\\
&=& \left\vert\det(A_\mathcal{M})\right\vert\left\vert\det(B_\mathcal{M})\right\vert\nonumber\\
&\le & \det(A_\mathcal{M})\prod\limits_{j=1}^{2k}\lambda_j^\downarrow(B),\label{eqmlid1}
\end{eqnarray}
where $\lambda_1^\downarrow(B),\ldots,\lambda_{2n}^\downarrow(B)$ denote the eigenvalues of $B$ arranged in decreasing order.
The last inequality in \eqref{eqmlid1} follows from Theorem III.1.5 of \cite{rbh}.
By using \eqref{eqprod}, we obtain
$$\prod\limits_{j=1}^{k}d_{i_j}^4(A\# B)\le \prod\limits_{j=1}^{k}d_{i_j}^2(A)\prod\limits_{j=1}^{2k}\lambda_j(B).$$
Let $M$ be a symplectic matrix such that $M^TBM=\diag(D(B),D(B)),$
where $D(B)$ is the positive diagonal matrix with diagonal entries $d_1(B)\le \cdots\le d_n(B).$
By the congruence invariance property of geometric means,
$M^T(A\# B)M=(M^TAM)\# (M^TBM),$
and the fact that for all positive definite matrices $X,$ $d_j(M^TXM)=d_j(X),$ $1\le j\le n,$
we can assume that $B$ is the diagonal matrix $\diag(D(B),D(B)).$
In this case, $\prod\limits_{j=1}^{2k}\lambda_j^\downarrow(B)=\prod\limits_{j=1}^{k}d_{n-j+1}^2(B).$
This gives the second inequality in \eqref{eqmlid}.

Take $G=A\# B.$ Then we have $A^{1/2}=GB^{-1/2}U^*.$
Calculations similar to those in \eqref{eqmlid1} give us the relation
$$\det(A_\mathcal{M})\le \left(\det(G_\mathcal{M})\right)^2\det(B^{-1}_{U\mathcal{M}})\le\left(\det(G_\mathcal{M})\right)^2\prod\limits_{j=1}^{2k}\lambda_j^\downarrow(B^{-1}).$$
Again, as in the previous paragraph, we can assume that $B$ is the diagonal matrix $\diag(D(B),D(B)),$ and by using \eqref{eqprod} we can obtain
$$\prod\limits_{j=1}^{k}d_{i_j}^2(A)\le\prod\limits_{j=1}^{k}d_{i_j}^4(G)\prod\limits_{j=1}^{k}\frac{1}{d_j^2(B)}.$$
This gives the first inequality of \eqref{eqmlid}.
\qed
\vskip.1in

We end the paper with an example which shows that unlike eigenvalues, the symplectic eigenvalues of $AA^T$ and $A^TA$ need not be equal.
Let $A$ be the $4\times 4$ matrix
$$A=\begin{bmatrix}
A_1 & O\\
O & A_2\end{bmatrix},$$
where $A_1=\begin{bmatrix}1 & 0\\
0 & 2\end{bmatrix}$ and $A_2=\begin{bmatrix}0 & 1\\
2 & 0\end{bmatrix}.$
Then straightforward calculations show that the symplectic eigenvalues of $A^TA$ are $2$ and $2$; whereas the symplectic eigenvalues of $AA^T$ are $1$ and $4.$
Further, it can be easily seen from Matlab experiments that for any two $2n\times 2n$ real positive definite matrices $A$ and $B,$
the symplectic eigenvalues of $A^{1/2}BA^{1/2}$ and $B^{1/2}AB^{1/2}$ need not be the same.
\vskip.2in

{\bf Acknowledgement}: The author acknowledges the
financial support from SERB
MATRICS grant number MTR/2018/000554.
\vskip.1in

\section*{Conflict of interest}

 The author declares that she has no conflict of interest.

\end{document}